\documentclass[12pt]{amsart}
\usepackage{verbatim,amssymb,latexsym,amscd}
\usepackage[all]{xy}
\usepackage[colorlinks,linkcolor=blue,citecolor=blue,urlcolor=blue]{hyperref}
\usepackage[T1]{fontenc}

\newcommand{\ie}{{\it i.e. }}
\newcommand{\cf}{{\it cf. }}
\newcommand{\eg}{{\it e.g. }}
\newcommand{\loccit}{{\it loc. cit. }}
\newcommand{\resp}{{\it resp. }}
\newcommand{\un}{\mathbf{1}}

\newcommand{\G}{\mathbb{G}}

\newcommand{\Q}{\mathbf{Q}}

\newcommand{\Z}{\mathbf{Z}}
\newcommand{\M}{\mathbf{M}}
\newcommand{\sA}{\mathcal{A}}
\newcommand{\sB}{\mathcal{B}}
\newcommand{\sC}{\mathcal{C}}
\newcommand{\sD}{\mathcal{D}}

\newcommand{\sI}{\mathcal{I}}

\newcommand{\sM}{\mathcal{M}}
\newcommand{\sN}{\mathcal{N}}

\newcommand{\sS}{\mathcal{S}}

\newcommand{\sV}{\mathcal{V}}

\newcommand{\Spec}{\operatorname{Spec}}

\newcommand{\Id}{\operatorname{Id}}

\newcommand{\Ker}{\operatorname{Ker}}
\newcommand{\Coker}{\operatorname{Coker}}
\newcommand{\IM}{\operatorname{Im}}

\newcommand{\abs}{{\operatorname{abs}}}

\newcommand{\tr}{{\operatorname{tr}}}

\newcommand{\End}{\operatorname{End}}

\newcommand{\Add}{{\operatorname{\bf Add}}}
\newcommand{\Ex}{{\operatorname{\bf Ex}}}

\newcommand{\Ab}{\operatorname{Ab}}

\newcommand{\car}{\operatorname{char}}

\newcommand{\rat}{{\operatorname{rat}}}

\newcommand{\num}{{\operatorname{num}}}

\newcommand{\rig}{{\operatorname{rig}}}

\newcommand{\tnil}{{\operatorname{tnil}}}
\newcommand{\Cat}{\operatorname{\bf Cat}}

\renewcommand{\Vec}{\operatorname{\bf Vec}}

\newcommand{\Rep}{\operatorname{\bf Rep}}

\newcommand{\by}{\xrightarrow}

\newcommand{\iso}{\by{\sim}}

\newcommand{\inj}{\hookrightarrow}

\newcommand{\surj}{\rightarrow\!\!\!\!\!\rightarrow}
 
\newcommand{\colim}{\varinjlim}
\renewcommand{\lim}{\varprojlim}

\newcommand{\fr}{{\operatorname{fr}}}

\renewcommand{\qed}{\hfill $\Box$\medskip}

\renewcommand{\phi}{\varphi}
\renewcommand{\epsilon}{\varepsilon}

\newcounter{spec}
\newenvironment{thlist}{\begin{list}{\rm{(\roman{spec})}}%
{\usecounter{spec}\labelwidth=20pt\itemindent=0pt\labelsep=10pt}}%
{\end{list}}%

\numberwithin{equation}{section}

\newtheorem*{Th}{Theorem}
\newtheorem{thm}{Theorem}[section]
\newtheorem{lemma}[thm]{Lemma}
\newtheorem{prop}[thm]{Proposition}
\newtheorem{cor}[thm]{Corollary}

\theoremstyle{definition}

\newtheorem{defn}[thm]{Definition}
\newtheorem{nota}[thm]{Notation}
\newtheorem{rk}[thm]{Remark}
\newtheorem{rks}[thm]{Remarks}

\newtheorem{ex}[thm]{Example}

\begin{document}
\title{A universal rigid abelian tensor category}
\author{Luca Barbieri-Viale}
\address{Dipartimento di Matematica ``F. Enriques'', Universit{\`a} degli Studi di Milano\\  Via C. Saldini, 50\\ I-20133 Milano\\ Italy }
\email{luca.barbieri-viale@unimi.it}
\author{Bruno Kahn}
\address{IMJ-PRG\\ Case 247\\4 place Jussieu\\
75252 Paris Cedex 05\\France}
\email{bruno.kahn@imj-prg.fr}

\begin{abstract}
We prove that any rigid additive symmetric monoidal category can be mapped to a rigid abelian symmetric monoidal category in a universal way. This sheds a new light on abelian $\otimes$-envelopes and on motivic conjectures such as Grothendieck's standard conjecture D and Voevodsky's smash nilpotence conjecture.
\end{abstract}
\maketitle

\section*{Introduction} Recently there has been a lot of activity on embedding a given rigid additive symmetric monoidal category into an abelian one in a universal way: we can quote on the one hand Coulembier's monoidal abelian envelopes from \cite{coul1} and his subsequent works alone and with coauthors \cite{coul2,coul3}, on the other O'Sullivan's super Tannakian hulls \cite{os}. Different but closely related is the work of Schäppi \cite{schappi}. We may also quote that of Delpeuch in the non-additive context \cite[App. B.3]{del}.

Both Coulembier and O'Sullivan impose the condition that their envelopes represent \emph{faithful} monoidal functors (for very good reasons, see introduction to Section \ref{s7}). In this paper, we take a different approach which relaxes this restriction, based on Freyd's free abelian category on a given additive category \cite{freyd}. Freyd's construction associates to an additive category $\sC$ an abelian category $\Ab(\sC)$ and a (fully faithful) additive functor $\iota_\sC:\sC\to \Ab(\sC)$ such that any additive functor from $\sC$ to an abelian category $\sA$ extends uniquely to an exact functor from $\Ab(\sC)$ to $\sA$. Suppose that $\sC$ is symmetric monoidal. In \cite{BVHP}, the authors provided $\Ab(\sC)$ with a right exact, symmetric monoidal structure such that $\iota_\sC$ is strong monoidal and universal for strong monoidal functors to abelian categories provided with a right exact symmetric monoidal structure (with a technical condition, see \loccit, Prop. 1.10). 

Suppose now that $\sC$ is rigid. We construct a localisation $T(\sC)$ of $\Ab(\sC)$ which is rigid and such that the induced functor $\sC\to T(\sC)$ is, this time, universal for strong additive symmetric monoidal (not necessarily faithful) functors from $\sC$ to rigid symmetric monoidal abelian categories: see Theorem \ref{t1}. When $\sC$ admits a $\otimes$-envelope, the latter turns out to be a localisation of $T(\sC)$ (Proposition \ref{l5}).

The novel thing here is that the ring of endomorphisms $Z(T(\sC))$ of the unit object of $T(\sC)$ need not be a field even if this is the case for $\sC$:  if $\sC$ is the category of representations of an affine group scheme $G$ over a field of characteristic $0$, then $Z(T(\sC))$ is a field if and only if $G$ is proreductive, see Example \ref{ex1}. This example is the only one where $T(\sC)$ can be computed so far (and only for certain $G$'s), besides Example \ref{ex4}.

Our motivating example was the one where $\sC$ is the $\Q$-linear category $\sM_\rat(k)$ of Chow motives over a field $k$ (motives with $\Q$ coefficients modulo rational equivalence). Write $T(k)$ for $T(\sM_\rat(k))$. By Jannsen's work \cite{jannsen}, the category $\sM_\num(k)$ of motives modulo numerical equivalence is abelian semi-simple, whence a canonical  $\Q$-linear exact $\otimes$-functor $T(k)\to \sM_\num(k)$. We prove in  Theorem \ref{t3} that this functor is a Serre localisation: it is an equivalence of categories if and only if  $Z(T(k))$ 
 is a field. In particular, the latter condition implies Grothendieck's standard conjecture D \cite[3.6 (i)]{klei}: homological and numerical equivalences agree for any Weil cohomology over $k$. 

In Proposition \ref{p8}, we give a related consequence of the existence of $T(k)$ on Schäppi's category mentioned earlier.

Conversely, Voevodsky's conjecture \cite[Conj. 4.2]{voe}, predicting that smash nilpotent and numerical equivalences agree, implies that $Z(T(k))$ is a field: see Proposition \ref{p5}. To summarise the situation, we have the following chain of implications, noting that the canonical functor $\sM_\rat(k)\to T(k)$ factors through $\sM_\tnil(k)$ by Theorem \ref{t1}, where $\tnil$ is smash-nilpotent equivalence:

\begin{Th} Voevodsky's conjecture $\iff$ $Z(T(k))$ is a field and $\sM_\tnil(k)\allowbreak\to T(k)$ is faithful $\Rightarrow$ $Z(T(k))$ is a field $\iff$ $T(k)\iso \sM_\num(k)$ $\Rightarrow$ the standard conjecture D.
\end{Th}

If one wants to stay away from any conjecture, one can argue that $T(k)$ gives in some sense an answer to Grothendieck's quest for a universal abelian category representing all cohomology theories on smooth projective varieties.  However, Grothendieck really thought of \emph{Weil cohomologies}, and $T(k)$ does not carry \emph{a priori} a grading: we plan to solve this issue in a further work, by adjoining such grading universally (and unconditionally).

\section{Notation and terminology} An additive, symmetric, monoidal, unital category (with bilinear tensor product) will be  briefly called a \emph{$\otimes$-category}. A \emph{$\otimes$-functor} between $\otimes$-categories is a strong symmetric, monoidal, unital additive  functor. See \cite{saa} or \cite{dm} for the background.

\begin{nota} For any $\otimes$-category $\sC$, we write $Z(\sC)$ for $\End_\sC(\un)$.
\end{nota}

Recall that the ring $Z(\sC)$ is commutative \cite[I.1.3.3.1]{saa} and that $\sC$ is a $Z(\sC)$-linear category; it coincides with the ring of \cite[III.5.d]{gabriel} (\cf \cite[I.2.5.2]{saa}). If $F : \sC \to \sD$ is a $\otimes$-functor, we write $Z(F) : Z(\sC) \to Z(\sD)$ for the induced ring homomorphism. 

\begin{nota} a) $\Add^\otimes$: the $2$-category of $\otimes$-categories, $\otimes$-functors and (additive) $\otimes$-natural isomorphisms.\\
b) $\Ex^\otimes$: the $2$-category of abelian $\otimes$-categories, exact $\otimes$-functors and (additive) $\otimes$-natural isomorphisms.\\
c) $\Add^\rig$ and $\Ex^\rig$: their $1$-full and $2$-full sub-$2$-categories of rigid categories.
\end{nota}

We have the following basic lemma:

\begin{lemma}\label{l0} For $\sC\in \Add^\otimes$, let $\sC_\rig$ denote its strictly full subcategory of dualisable objects. Then $\sC_\rig$ is closed under direct sums and direct summands, and $F(\sC_\rig)\subset \sD_\rig$ for any $\otimes$-functor $F:\sC\to \sD$.
\end{lemma}

\begin{proof} Both points follow from \cite[Th. 1.3]{dp}.
\end{proof}

\section{Reminders and complements on rigid abelian $\otimes$-categories} Let $\sA$ be a rigid abelian $\otimes$-category.

\begin{lemma}\label{l2} The $\otimes$-structure of $\sA$ is exact.
\end{lemma}

This is \cite[Prop. 1.16]{dm}.\qed
\begin{prop}\label{l4} a) If $U$ is a subobject of $\un$, then $\un=U\oplus U^\perp$ where $U^\perp = \Ker(\un\to U^\vee)$, and $U\otimes U = U$.\\
b) There is a bijective correspondence between subobjects of $\un$, idempotents $e\in Z(\sA)$ and  decompositions
$\sA=\sA_1\times \sA_2$ in which an object is in $\sA_1$ (\resp in $\sA_2$) if $e$ (\resp
$1-e$) acts as the identity morphism on it.
\end{prop}

\begin{proof}
a) The first fact is \cite[Prop. 1.17]{dm}, and  the second one is contained in its proof. b) is \cite[Rem. 1.18]{dm}.
\end{proof}

\begin{lemma}\label{l10} Let $\sA=\sA_1\times \sA_2$. Any $\otimes$-functor $F: \sA \to \sB$ such that $Z(\sB)$ is a field vanishes either on $\sA_1$ or $\sA_2$.
\end{lemma}
\begin{proof}
Indeed, the idempotent corresponding to this decomposition in Proposition \ref{l4} b) must go to $0$ or $1$ in $Z(\sB)$.
\end{proof}

\begin{prop}\label{l3} Suppose that $Z(\sA)$ is a field. Then a $\otimes$-functor from $\sA$ to a (nonzero) abelian $\otimes$-category $\sB$ is faithful if it is exact. The converse is true if $Z(\sB)$ is a field.
\end{prop}

\begin{proof}
``If'' is \cite[Prop. 1.19]{dm}, and ``only if'' is \cite[Th. 2.4.1]{coul2} (see also \cite[Lemma 10.7]{os} in the case of super Tannakian categories).
\end{proof}

\begin{prop}\label{p2} a) if $Z(\sA)$ is a field, $\sA$ is integral: for two morphisms $f,g$, $f\otimes g=0$ implies $f=0$ or $g=0$.\\
b) In general, $\sA$ is reduced: for a morphism $f$, $f^{\otimes 2}=0$ implies $f=0$. In particular, the ring $Z(\sA)$ is reduced.\\
c) $\sA$ is fractionally closed in the sense of \cite[Sect. 3, bot. p. 6]{os}. In particular, $Z(\sA)$ is its own total ring of quotients.\\
d) If $Z(\sA)$ is a domain then it is a field.\\
e) $\un$ is Noetherian (equivalently Artinian) if and only if  $Z(\sA)$ is, and then $\sA$ is equivalent to  a product $\prod_{i\in I} \sA_i$, where $I$ is finite and $Z(\sA_i)$ is a field for each $i$.\\
f) $Z(\sA)$ is absolutely flat (= von Neumann regular).
\end{prop}

\begin{proof} a) is \cite[Rem. 2.10]{exandfaith}. b) was suggested to the second author by Peter O'Sullivan. By rigidity, we may assume that the domain of $f$ is $\un$. Then $f$ factors through a monomophism $\Coker(f)\inj A$, where $A$ is the codomain of $f$. If $f\ne 0$, $\Coker(f)\ne 0$; but this is a subobject of $\un$ by Proposition \ref{l4} a), hence $\Coker(f)^{\otimes 2} = \Coker(f)\ne 0$ by the same Proposition, and $f\otimes f\ne 0$ by Lemma \ref{l2}.  c) is \cite[Lemma 3.1]{os}. d) is special case of c), but we give a direct proof: if $Z(\sA)$ is not a field, it contains an idempotent by Proposition \ref{l4} b) so it cannot be a domain. e) follows easily from the same proposition.
The consequences of b) and c) on $Z(\sA)$ follow from the fact that, in this ring, composition and tensor product coincide, see  \cite[I.1.3.3.1]{saa}.
For f), see \cite[Prop. 3.2]{krig}.   
\end{proof}

\section{Complements on abelian $\otimes$-categories}

In this section, $\sA$ is an abelian (not necessarily rigid) $\otimes$-category. 
The following proposition completes the proof of \cite[Prop. 1.13 (3)]{BVHP}, removing its hypothesis of right exactness of the Homs.

\begin{prop}\label{p3.1} If the tensor structure of $\sA$ is exact, then the full subcategory $\sA_\rig$ of dualisable objects is closed under kernels and cokernels. In particular, $\sA_\rig$ is abelian and the inclusion functor $\sA_\rig\inj \sA$ is exact.
\end{prop}

\begin{proof} Let $A'\to A\to A''\to 0$ be exact, with $A',A\in \sA_\rig$, and write $K=\Ker(A^\vee\to {A'}^\vee)$. For any $B,C\in \sA$, we have a commutative diagram of exact sequences
\[\begin{CD}
0\to \sA(A''\otimes B,C)@>>> \sA(A\otimes B,C)@>>> \sA(A'\otimes B,C)\\
@V\wr VV @V\wr VV\\
0\to \sA(B,K\otimes C)@>>> \sA(B,A^\vee\otimes C)@>>> \sA(B,{A'}^\vee\otimes C)
\end{CD}\]
due to the exactness of $\otimes$. It induces an isomorphism $\sA(A''\otimes B,C)\iso \sA(B,K\otimes C)$ natural in $B$ and $C$, showing that $K$ is right dual to $A''$. But then, it is also left dual since $\otimes$ is symmetric (see \cite{brug}).

For an exact sequence $0\to A'\to A\to A''$ with $A,A''\in \sA_\rig$, we argue similarly by changing the side of the tensor products in the Hom groups.
\end{proof}

\begin{prop}\label{p3.2} Suppose $\sA$ essentially small with a right exact tensor structure. Let $\sI\subseteq \sA $ be a Serre subcategory containing the objects  $K(f,A):=\Ker(1_A\otimes f)$ for all monomorphisms $f$ and stable under (external) tensor products. Then $\sA/\sI$ inherits a tensor structure, which is exact. If $\sI$ is minimal for those properties, $\sA\to \sA/\sI$ is initial for exact $\otimes$-functors from $\sA$ to an abelian $\otimes$-category with exact tensor structure.
\end{prop}

\begin{proof} 
Let $\Sigma$ be the multiplicative system associated to $\sI$, \ie $s\in \Sigma$ $\iff$ $\Ker s, \Coker s\in \sI$. The first point means that $\Sigma$ is stable under tensor product with identities on the left and on the right, so, say, on the left by symmetry. We proceed as follows:

1) Let $s:C\inj D$ be a monomorphism in $\Sigma$, and let $A\in \sA$. The exact sequence 
\[0\to K(s,A)\to A\otimes C\by{1_A\otimes s} A\otimes D\to A\otimes \Coker s\to 0\]
shows that $1_A\otimes s\in \Sigma$.

2) Let $s:C\surj D$ be an epimorphism in $\Sigma$, and let $A\in \sA$. The exact sequence
\[A\otimes \Ker s\to A\otimes C\by{1_A\otimes s} A\otimes D\to 0\]
shows  that $1_A\otimes s\in \Sigma$.

3) Any $s\in \Sigma$ can be written $t\circ u$, with $u$ epi, $t$ mono and $t,u\in \Sigma$. This, with 2) and 3), completes the proof of the first point.

Let us now prove exactness of the tensor product. Let $A\in \sA/\sI$ and let $(*): 0\to C'\by{f} C\by{g} C''\to 0$ be a short exact sequence in $\sA/\sI$ (recall that $\sA$ and $\sA/\sI$ have the same objects). By \cite[p. 368, Cor. 1]{gabriel}, $(*)$ is isomorphic to a short exact sequence of $\sA$; to show that $A\otimes (*)$ is exact, we may therefore assume that $f$ and $g$ come from morphisms of $\sA$. Since the projection functor $\sA\to \sA/\sI$ is exact, we already have right exactness. But $1_A\otimes f$ is a monomorphism in $\sA/\sI$ by definition of $\sI$, so the proof is complete.

Note that a minimal $\sI$ exists: any intersection of Serre subcategories (\resp closed under external tensor product) is a Serre subcategory (\resp is closed under external tensor product). The initiality is now clear, since any exact $\otimes$-functor from $\sA$ to an abelian $\otimes$-category sends all objects $K(f,A)$ to $0$.
\end{proof}

\begin{rk} Let $\sI_0$ be the minimal Serre subcategory as in Proposition \ref{p3.2}, and let $\sI'\subseteq \sI_0$ denotes the smallest Serre subcategory of $\sA$ containing the objects  $K(f,A)$ (no $\otimes$-ideal condition). It is tempting to try and show that $\sI'=\sI_0$. At least, $B\otimes K(f,A)\in \sI'$ for any $A,B\in \sA$ and any monomorphism $f:C\inj D$, as follows from the exact sequence
\[0\to K(g,B) \to B\otimes K(f,A)\to K(f,B\otimes A)\]
where $g$ is the monomorphism $K (f, A)  \inj A \otimes C$. But this does not seem sufficient: since $\otimes$ is right exact, $\sC\otimes\sI'$ is stable under quotients and extensions, but maybe not under kernels.
\end{rk}

\begin{rk}  In general, $Z(\sA/\sI)$ need not be a field even if $Z(\sA)$ is: see Example \ref{ex1} below. Here are two things one can say:\\
a) We have
\begin{equation}\label{eq3}
Z(\sA/\sI) = \colim_{A',A''} \sA(A',\un/A'') 
\end{equation}
where $A'$ (\resp $A''$ runs through subobjects of $\un$ such that $\un/A'\in \sI$ (\resp $A''\in \sI$): this translates the definition of morphisms in $\sA/\sI$ as in \cite[III.1]{gabriel}.\\
b) $\un_{\sA/\sI}$ is irreducible $\iff$ any subobject of $\un_\sA$ belongs to $\sI$ $\iff$ any quotient of $\un_\sA$ belongs to $\sI$: this follows from \cite[p. 368, Cor. 1]{gabriel} which was already used in the proof of Proposition \ref{p3.2}. In particular, $\un_\sA$ irreducible $\Rightarrow$ $\un_{\sA/\sI}$ irreducible.
\end{rk}

\begin{prop}\label{p4} Let $\sA\in \Ex^\otimes$ be essentially small and let $\sI$ be a Serre subcategory stable under external tensor product. If the tensor structure of $\sA$ is exact, there is a unique tensor structure on $\sA/\sI$ such that $p:\sA\to \sA/\sI$ is a $\otimes$-functor, and this tensor structure is exact. If $\sA$ is rigid, so is $\sA/\sI$, and $Z(\sA)\to Z(\sA/\sI)$ is surjective.
\end{prop}

\begin{proof} The first two points are special cases of Proposition \ref{p3.2}, since all objects $K(f,A)$ are $0$ by assumption. The rigidity of $\sA/\sI$ follows from Lemma \ref{l0}, since $p$ is surjective.  The last point follows from \eqref{eq3} and Proposition \ref{l4}, which implies that $Z(\sA)\to \sA(A',\un/A'')$ is (split) surjective for all $A',A''\subseteq \un$.
\end{proof}

\section{Freyd's universal construction} 

Recall   (\cite{freyd}, see also \cite[2.10]{krause}, \cite[Th. 4.1 and Cor. 4.2]{prest},  \cite[\S 1]{BVHP}) that, for any additive category $\sC$, there is an abelian category $\Ab(\sC)$ and an additive functor $\iota_\sC:\sC\to \Ab(\sC)$ such that any additive functor $\sC\to \sA$, where $\sA$ is an abelian category, extends through $\iota_\sC$ to an exact functor $\Ab(\sC)\to \sA$, unique up to unique equivalence of categories. The functor $\iota_\sC$ is fully faithful.

\begin{lemma}\label{l1}
a) The category $\sC$ generates $\Ab(\sC)$ as an abelian category in the following sense: if $\sA\subseteq \Ab(\sC)$ is a strictly full abelian subcategory such that the inclusion functor is exact and $\sA$ contains $\iota_\sC(\sC)$, then $\sA=\Ab(\sC)$. In particular, any object of $\Ab(\sC)$ is a subquotient of $\iota_\sC (C)$ for some $C\in \sC$.\\
b) $Z(\sC)\iso Z(\Ab(\sC))$.
\end{lemma}

\begin{proof} a) is \cite[ Lemma 4.12]{prest}, and b) is obvious by full faithfulness. In a) ``In particular'' holds because any subobject and any quotient of a subquotient is a subquotient. 
\end{proof}

\begin{defn}\label{r4} An abelian category $\sA$ is \emph{split} if $\sA\by{\iota_\sA} \Ab(\sA)$ is an equivalence of categories.
\end{defn}

\begin{prop}\label{p15} For an abelian category $\sA$, the following are equivalent:
\begin{enumerate}
\item $\sA$ is split.
\item Every additive functor from $\sA$ to an abelian category is exact.
\item Every short exact sequence splits.
\item Every object is projective.
\item Every object is injective.
\end{enumerate}
This holds in particular if $\sA$ is semisimple\footnote{Here we adopt the terminology of \cite[2.1.1]{AK2}: a preadditive category $\sA$ is semisimple if every left $\sA$-module is a direct sum of simple objects. If $\sA$ is abelian, this means  \cite[A.2.10 (10)]{AK2} that every object of $\sA$ is a finite direct sum of simple objects.}.\\
If $\sA$ is split, any full pseudo-abelian subcategory of $\sA$ is a Serre subcategory (in particular, abelian) and is split, and any Serre localisation of $\sA$ is split. The same holds when replacing ``split'' by ``semisimple''.
\end{prop}

\begin{proof}
The only possibly nonobvious point is (2) $\Rightarrow$ (4): use the Hom functor.
\end{proof}

If $\sC$ has a $\otimes$-structure, then $\Ab(\sC)$ inherits a right exact $\otimes$-structure for which $\iota_\sC$ is a  $\otimes$-functor \cite[Prop. 1.8]{BVHP}.

\begin{prop}\label{p1}  Let $\sI\subseteq \Ab(\sC)$ be minimal in Proposition \ref{p3.2}, and write $T(\sC)=\Ab(\sC)/\sI$. Then the composition  $\sC\by{\iota_\sC} \Ab(\sC)\to T(\sC)$ is $2$-universal for  $\otimes$-functors from $\sC$ to abelian $\otimes$-categories with exact tensor structure (with respect to exact  $\otimes$-functors). In particular, $\sC$ generates $T(\sC)$ in the same sense as in Lemma \ref{l1} a).
\end{prop}

\begin{proof} Let $F:\sC\to \sA$ be a $\otimes$-functor, with $\sA\in \Ex^\otimes$. If the tensor structure of $\sA$ is exact, $F$ factors uniquely through an exact $\otimes$-functor $\tilde F:\Ab(\sC)\to \sA$. Indeed, we apply \cite[Prop. 1.10]{BVHP}: by the exactness of the tensor product, we may take $\sA^\flat=\sA$ in \loccit so its hypothesis is trivially verified. Then $\tilde F$ factors uniquely through $T(\sC)$ by  Proposition \ref{p3.2}.
\end{proof}

\begin{prop}\label{r3} Let $\sA\in\Ex^\rig$. Then $\sA$ is split if and only if $\un$ is projective.
\end{prop}

\begin{proof}  If $\un$ is projective, the functor $A\mapsto \sA(B,A)\simeq \sA(\un,B^\vee\otimes A)$ is right exact, since $\otimes$ is exact by Lemma \ref{l2}.
\end{proof}

\section{Main theorem}

\begin{thm}\label{t1} Let $\sC$ be a rigid additive $\otimes$-category. Then the $2$-functor
\[\sA\mapsto \Add^\otimes(\sC,\sA)\]
from $\Ex^\rig$ to $\Cat$ is $2$-representable by the category $T(\sC)$ of Proposition \ref{p1}. Moreover, the obvious functors
\[T(\sC)\to T(\sC/\sqrt[\otimes]{0})\to T((\sC/\sqrt[\otimes]{0})_\fr)\to T(((\sC/\sqrt[\otimes]{0})_\fr)^\natural)\]
are equivalences of categories, where $\sqrt[\otimes]{0}$ is the $\otimes$-ideal of $\otimes$-nilpotent morphisms \cite[Def. 7.4.1]{AK2}, $\sD_\fr$ (\resp $\sD^\natural$) is the fractional closure of a $\otimes$-category $\sD$ \cite[p. 8]{os} (\resp its pseudo-abelian envelope).
\end{thm}

\begin{proof} Let $T(\sC)_\rig$ be the strictly full subcategory of dualisable objects:  by Proposition \ref{p3.1}, it is abelian and the full embedding $T(\sC)_\rig\inj T(\sC)$ is exact.  Since $\sC$ is rigid, its image in $T(\sC)$ lands into $T(\sC)_\rig$ by Lemma \ref{l0}; therefore, $T(\sC)_\rig=T(\sC)$ by Proposition \ref{p1}.  Let now $\sA\in \Ex^\rig$. Its tensor structure is exact by Lemma \ref{l2}, hence, again  by Proposition \ref{p1}, any $\otimes$-functor $\sC\to \sA$ factors through $T(\sC)$, uniquely up to unique $\otimes$-equivalence.

In the last claim, the first equivalence follows from Proposition \ref{p2} b) and the second (\resp third) one from the fact that rigid abelian $\otimes$-categories are fractionally closed \cite[Lemma 3.1]{os} (\resp that abelian categories are pseudo-abelian).
\end{proof}

\begin{cor}\label{r1}
If $\sC$ is abelian in Theorem \ref{t1},  the canonical functor $\sC\to T(\sC)$ has an exact $\otimes$-retraction $\sigma_\sC$. If moreover $\sC$ is split, this functor is an equivalence of $\otimes$-categories.
\end{cor}

\begin{proof} The first claim follows from the universal property of $T(\sC)$ applied to $\Id_\sC:\sC\to \sC$. The second one follows from the definition of split (Definition \ref{r4}).
\end{proof}

\begin{rks}  a) One can presumably extend Theorem \ref{t1} to not necessarily rigid additive $\otimes$-categories by using  \cite[App. B.3]{del}.\\
b) For $\sC$ as in Theorem \ref{t1}, $\Ab(\sC)$ is rigid if and only if its tensor structure is exact. Necessity follows from Lemma \ref{l2}; if conversely the tensor structure of $\Ab(\sC)$ is exact, then $\Ab(\sC)\iso T(\sC)$ which is rigid by (the proof of) Theorem \ref{t1}.\\
c) The category $(\sC/\sqrt[\otimes]{0})_\fr)^\natural$ of Theorem \ref{t1} is reduced, fractionally closed and pseudo-abelian.\footnote{We thank Peter O'Sullivan for confirming that the fractional closure of a reduced $\otimes$-category is reduced.}\\
\end{rks}

\begin{ex}\label{ex2} Suppose that, in $\sC$, there exists a nilpotent endomorphism with nonzero trace. Then $T(\sC)=0$ because, in rigid abelian $\otimes$-categories, any nilpotent endomorphism has trace $0$ \cite[Prop. 7.3.3]{AK2}. An example where this happens is given in \cite[\S 5.8]{dS}.
\end{ex}

\begin{ex}\footnote{This example was found independently by P. O'Sullivan \cite{os4}.}\label{ex4}  Let $R^+$ be the additive completion of a commutative ring $R$ considered as a preadditive category (the objects of $R^+$ are $R^n$), provided with its canonical $\otimes$-structure. By \cite[Prop. 5]{olivier}, there exists a ring homomorphism $R\to R^\abs$ which is universal for homomorphisms from $R$ to absolutely flat rings. We claim that $T(R^+)= R^\abs\text{-mod}$, where the latter is the (abelian) rigid $\otimes$-category of finitely presented (equivalently, finitely generated projective) $R^\abs$-modules. If $R$ is absolutely flat, this follows from \cite[Prop. 10.2.38]{prest3} and Remark 5.3 b), the first reference showing that $\Ab(R^+)\iso R\text{-mod}$ in this case.  In general, we use the fact that $Z(\sA)$ is absolutely flat for any $\sA\in \Ex^\rig$, Proposition \ref{p2} f). For such $\sA$, a $\otimes$-functor $F:R^+\to \sA$ amounts to an $R$-module structure on $\un_\sA$.  Thus $Z(\sA)$ is an $R^\abs$-algebra, so $F$ factors uniquely through $(R^\abs)^+$, hence through $R^\abs\text{-mod}$ as promised.
\end{ex}

\section{Comparisons}\label{s7}

\subsection{Abelian $\otimes$-envelopes}
Let $\Add^\rig_f$ be the sub-$2$-category of $\Add^\rig$ restricted to the $\sC$'s such that $Z(\sC)$ is a field and to \emph{faithful} functors. Similarly, let $\Ex^\rig_f$ be the 1-full, 2-full sub-$2$-category of $\Ex^\rig$ determined by those $\sA\in \Ex^\rig$ such that $Z(\sA)$ is a field; note that in $\Ex^\rig_f$ all exact $\otimes$-functors are automatically faithful by Proposition \ref{l3}, so $\Ex^\rig_f$  is contained in $\Add^\rig_f$.  Coulembier as well as O'Sullivan consider the universal property of Theorem \ref{t1} restricted to these $2$-categories.  This has an advantage and a drawback:

\begin{itemize}
\item  By Proposition \ref{l3},  a solution to this universal problem is automatically an envelope (an idempotent construction).
\item Such a solution is much more difficult to construct (when it exists, see Example \ref{ex2}).
\end{itemize}

Nevertheless, both authors provide a solution in special cases, by very different methods.

To formalise things, let us set up a definition:

\begin{defn}\label{df77}
Let $\sC\in \Add^\rig_f$. An \emph{abelian $\otimes$-envelope} of $\sC$ is a category $E(\sC)\in \Ex^\rig_f$ which $2$-represents the $2$-functor 
\[\sA\mapsto \Add^\rig_f(\sC,\sA)\]
from $\Ex^\rig_f$ to $\Cat$.
\end{defn}

Note that $Z(E(\sC))$ must be a field extension of $Z (C)$ by definition. If $\sC\to E (\sC)$ is also full then we have
\begin{equation}\label{eq.coul}
Z(E(\sC)) = Z (\sC). 
\end{equation}

\begin{prop}\label{l5} a) $E (\sC)$ exists if and only if 
\begin{thlist}
\item there exists a faithful $\otimes$-functor $F: \sC\to \sA$ with $\sA\in \Ex^\rig_f$;
\item the (Serre) kernel $\sS$ of the induced functor $T (\sC)\to\sA$ does not depend on $(\sA, F)$.  
\end{thlist}
In this case, $\sC$ is integral and $\sC\to T (\sC)$ is faithful. \\
b) Any abelian $\sC$ is its own abelian $\otimes$-envelope. 
\smallskip

In particular, if $E(\sC)$ exists  we have a canonical localisation $\otimes$-functor
\begin{equation}\label{eq2}
T(\sC)\to E(\sC).
\end{equation}
\end{prop}
\begin{proof} 

a) Conditions (i) and (ii) are obviously necessary. Conversely,
assume (i) and (ii). Then $E (\sC)= T (\sC)/\sS$ is rigid by Proposition \ref{p4}. Moreover, $Z(E (\sC))$ is a subring of the field $Z(\sA)$ and therefore it is a field by Proposition \ref{p2} d). Thus $E (\sC)\in \Ex^\rig_f$.  Clearly, any functor $F$ as in (i) factors uniquely through $E(\sC)$, so $E(\sC)$ satisfies the universal property. Finally, the induced functor $\sC\to E (\sC)$ is faithful because its composition with $E (\sC)\to \sA$ is faithful for $\sA$ as in (i).

The integrality of $\sC$ follows from Proposition \ref{p2} a) and the faithfulness is obvious. 

b) As explained above, this follows from Proposition \ref{l3}. 

The last remark follows from the proof of a).
\end{proof}

Suppose that $E(\sC)$ exists. 
Then \eqref{eq2} is an equivalence if and only if \emph{any} additive $\otimes$-functor from $\sC\in \Add^\rig_f$ to $\sA\in \Ex^\rig$ is faithful. For example, the category of representations of an affine group scheme $G$ is abelian hence its own envelope by Proposition \ref{l5} b), but \eqref{eq2} is not an equivalence if $G$ is not proreductive by Example \ref{ex1} below. 
 
\begin{prop}\label{p6} a) If $E(\sC)$ exists, \eqref{eq2} is an equivalence of categories if and only if $Z(T(\sC))$ is a field.\\
b) Suppose that $Z(T(\sC))$ is a field. Then $T(\sC)$ is the abelian $\otimes$-envelope of $\sC/I$, where $I$ is the (additive) kernel of $\sC\to T(\sC)$. In particular,  $T(\sC)$ is an abelian $\otimes$-envelope of $\sC$ if and only if $\sC\to T(\sC)$ is faithful. 
\end{prop}

\begin{proof} All this follows once again from Proposition \ref{l3}.
\end{proof}

\subsection{Coulembier's work}
Coulembier's condition for an envelope \cite[Def. 1.3.4]{coul1} is Definition \ref{df77} plus the requirement that \eqref{eq.coul} holds. He proves:

\begin{thm}[\protect{\cite[Th. A]{coul1}}]\label{t5} Let $\sC\in \Add^\rig_f$. Then $E(\sC)$ exists, with property \eqref{eq.coul}, provided every morphism $f$ in $\sC$ is split by a strongly faithful object in $\sC$. 
\end{thm}

Here, $X\in \sC$ is \emph{strongly faithful} if $X \otimes -: \sC \to \sC$ reflects all kernels and cokernels in $\sC$, and a morphism $f:X\to Y$ in $\sC$ is \emph{split} if there exists $g:Y\to X$ such that $f\circ g\circ f=f$. ``Split by $X$'' means that $1_X\otimes f$ is split.

See \cite[Th. 4.1.1 (a)]{coul1} for another sheaf-theoretic sufficient condition.

\subsection{O'Sullivan's work} Let $\sC\in \Add^\rig$ be essentially small, $\Q$-linear, integral (see Proposition \ref{p2} a)) and Schur-finite. (In particular, the integrality hypothesis implies that $Z(\sC)$ is a field.) According to \cite[Def. 10.2]{os}, we call such a category \emph{pseudo-Tannakian}, and \emph{super Tannakian} if it is abelian. 

Write $\Add^\rig_t$ for the $1$-full and $2$-full subcategory of $\Add^\rig_f$ formed of pseudo-Tannakian categories, and $\Ex^\rig_t$ or the $1$-full and $2$-full subcategory of $\Ex^\rig$ formed of super Tannakian categories. We have:

\begin{thm}\label{t4} For any $\sC\in \Add^\rig_t$, the $2$-functor
\[\sA\mapsto \Add^\rig_t(\sC,\sA)\]
from $\Ex^\rig_t$ to $\Cat$ is $2$-representable.
\end{thm}

\begin{proof} This is \cite[Lemma 10.7 and Th. 10.10]{os}.
\end{proof}

\begin{rk} Let $ST(\sC)$ be the solution of the above universal problem.  It can be proven that $ST(\sC)=E(\sC)$. See \cite[Thm. 5.5.]{krig}\end{rk}

\begin{rks} a) As pointed out by O'Sullivan, any Kimura category verifying the conditions of Theorem \ref{t5} is semi-simple.\\ 
b) As far as we know,  and in spite of Propositions \ref{l5} and \ref{p6}, Theorem \ref{t1} does not imply either Theorem \ref{t5} or Theorem \ref{t4} in any obvious way!
\end{rks}

\section{The split quotient of $T (\sC)$}

\subsection{The ideal $\sN$}
Let $\sC\in \Add^\rig$.  Recall from \cite[7.1]{AK2} the $\otimes$-ideal $\sN_{\sC}\subseteq \sC$ of morphisms universally of trace $0$: for $A,B\in \sC$, 
\[\sN_\sC(A,B)=\{f\in \sC(A,B)\mid \tr(gf)=0\ \forall g\in \sC(B,A)\}\]
where $\tr$ is the categorical trace. 

\begin{lemma}\label{l9} For any split $\sA\in \Ex^\rig$ (Definition \ref{r4}), we have $\sN_\sA=0$. Conversely, if $\sN_\sA=0$ and $Z(\sA)$ is Noetherian, then $\sA$ is split.
\end{lemma}

\begin{proof} By \cite[6.1.5]{AK2}, it suffices to show that $\sN_\sA(\un,A)=0$ for any $A\in \sA$. Let $f:\un\to A$ be a nonzero morphism. If $U=\IM f$, $U$ is injective (Proposition \ref{p15} (5)), hence the induced monomorphism $U\to A$ has a retraction. Since $U$ is a direct summand of $\un$, this retraction yields a morphism $g:A\to \un$ such that $gf\ne 0$.

For the converse, it suffices by Proposition \ref{r3} to show that $\un$ is projective. Let $f:A\to \un$ be an epimorphism, with $A\in \sA$. Since $\sN(A,\un) = 0$, there is $g:\un \to A$ such that $fg\ne 0$. Let $U\ne 0$ be the image of $fg$.  Replace $f$  by $f_1=(1-e)f$, where $e$ is the idempotent with image $U$ in the decomposition $\un=U\oplus U^\perp$ of Proposition \ref{l4} a). Since $f$ was epi, $f_1$ is epi on $U^\perp$, hence nonzero if $U^\perp\ne 0$. Using $\sN(A,U^\perp) = 0$,  we find $g_1:U^\perp\to A$ such that $g_1f_1 \ne 0$. Iterating, we get a strictly increasing sequence of subobjects of $\un$, which must stop at $\un$ at a finite step. Collecting everything, we get a section of $f$.
\end{proof}

\begin{lemma}\label{l8}
Let $F : \sC\to \sD$ be a full $\otimes$-functor with $\sC, \sD\in \Add^\rig$. 
Let $f$ be a morphism of $\sC$. Then $f\in \sN_\sC$ $\Rightarrow$ $F(f)\in \sN_\sD$, and the converse is true if $Z(F):Z(\sC)\to Z(\sD)$ is injective. Under this condition, the induced full $\otimes$-functor 
\[ \sC/\sN_\sC \to \sD/\sN_\sD\]
is faithful. 
\end{lemma}
\begin{proof}
Since
\[\tr(F(f))=F(\tr(f))\]
for any $f\in \sC$, this lemma is trivial.
\end{proof}

\begin{prop}\label{p7} Let $\sC\in \Add^\rig$ be such that $Z(\sC)$ is a field. Then the following conditions are equivalent:
\begin{thlist}
\item $(\sC/\sN_\sC)^\natural$ is abelian.
\item $(\sC/\sN_\sC)^\natural$ is (abelian and) split.
\item There exists $F:\sC\to \sD$ as in Lemma \ref{l8}, with $\sD$ split.
\end{thlist}
\end{prop}

\begin{proof}(i) $\Rightarrow$  (ii) follows from the second part of Lemma \ref{l9}, since $Z((\sC/\sN_\sC))^\natural)\allowbreak =Z(\sC)$ is a field. (ii) $\Rightarrow$ (iii) is trivial. If (iii) holds,  the hypotheses imply that  $Z(F)$ is injective and $\sN_\sD=0$ (by the first part of Lemma \ref{l9}), thus Lemma \ref{l8} gives us a fully faithful $\otimes$-functor $\sC/\sN_\sC \to \sD$. Since $\sD$ is pseudo-abelian, it extends to a full embedding $(\sC/\sN_\sC)^\natural\inj \sD$. Finally, $(\sC/\sN_\sC)^\natural$ is abelian and split thanks to Proposition \ref{p15}. So (iii) $\Rightarrow$ (ii). 
\end{proof}

\begin{rk}\label{r2} Proposition \ref{p7} means that $\sC\to (\sC/\sN_\sC)^\natural$ is universal with respect to full $\otimes$-functors to split abelian $\otimes$-categories -- assumimg such functors exist, which fails \eg in Example \ref{ex2}. The situation is parallel to that of Proposition \ref{l5}. When $\sC$ is the category of pure motives over a field (see Section \ref{s8} below), we recover the classical fact that the Hodge conjecture or the Tate conjecture (plus semi-simplicity) implies the standard conjecture D.
\end{rk}

\subsection{A splitting} 

In this subsection, we assume that $(\sC/\sN_\sC)^\natural$ is abelian split. We write $S(\sC)$ for this category.

\begin{ex}\label{ex3} By Lemma \ref{l0} and  \cite[Th. 1 a)]{AK3},  the above hypothesis is satisfied provided $K=Z(\sC)$ is a field and  there exists an extension $L/K$ and a $K$-linear $\otimes$-functor $H : \sC\to \sV$ to a nonzero rigid $L$-linear abelian $\otimes$-category $\sV$ in which Hom groups have finite $L$-dimension.  In this case, $S(\sC)$ is even semisimple.
\end{ex}

Write $\pi$ for the $\otimes$-functor $\sC\to S(\sC)$. By  Theorem \ref{t1}, $\pi$ induces an exact $\otimes$-functor 
\begin{equation}\label{eq4}
\bar \pi:T(\sC)\to S(\sC).
\end{equation}
Let $T_0(\sC)$ be the quotient $T(\sC)/\Ker \bar\pi$ and let
$$\bar \pi_0:T_0(\sC)\to S(\sC)$$
be the induced  faithful exact $\otimes$-functor. Note that $T_0(\sC)$ is still rigid by Proposition \ref{p4}.
\begin{thm}\label{t2}  $\bar \pi_0$ is an equivalence of categories. 
\end{thm}

We first prove that $\bar \pi_0$ is full. By rigidity, this is equivalent to

\begin{lemma} \label{l7} The injection 
\[T_0(\sC)(\un , X) \by{\bar\pi_0} S(\sC)(\un , \bar \pi_0 (X))\]
is surjective for any $X\in T_0(\sC)$.
\end{lemma}

\begin{proof} Let $0\to X'\to X\to X''\to 0$ be a short exact sequence in $T_0(\sC)$. We then get the following commutative diagram 
\[\begin{CD}
0\to T_0(\sC)(\un , X')@>>> T_0(\sC)(\un , X)@>>> T_0(\sC)(\un , X'')\\
@V\bar\pi_0VV@V\bar\pi_0 VV @V\bar\pi_0VV\\
0\to S(\sC)(\un , \bar \pi_0 (X'))@>>> S(\sC)(\un , \bar \pi_0 (X))@>>> S(\sC)(\un , \bar \pi_0 (X''))\to 0
\end{CD}\]
with exact rows and vertical injections. The bottom exact row follows from the exactness of $\bar \pi_0$ and  the splitness of $S(\sC)$, granting that $S(\sC) (\un,-)$ is right exact.
If the middle vertical arrow is an epimorphism then the left-most one is by diagram chase; moreover, the right-most is an epimorphism. Therefore, the statement is stable under passing to subquotients: since it is true for objects coming from $\sC$, it is true in general by Proposition \ref{p1}.
\end{proof}

To conclude the proof of Theorem \ref{t2}, we show that $\bar \pi_0$ is essentially surjective. Let $Y=(\pi(C),e)\in S(\sC)$, where $C\in \sC$ and $e$ is an idempotent of $\End_{\sC/\sN_\sC}(\pi(C))$. By the full faithfulness of $\bar \pi_0$, $e$ lifts to an idempotent of $p(C)$, where $p:\sC\to  T_0(\sC)$ is the canonical functor. \qed

\begin{cor}\label{c2}Let $F:\sC\to\sA$ be a $\otimes$-functor with $\sA\in \Ex^\rig$. Assume that $Z(\sC)$ is a field. If  $F$ is full and $\sA$ is split,  then the functor $$\bar F:T(\sC)\to \sA$$ induced from Theorem \ref{t1} factors through $S(\sC)$ and $\Ker\bar F = \Ker\bar \pi$.
\end{cor}

\begin{proof} This follows from Remark \ref{r2} and Theorem \ref{t2}. 
\end{proof}

\begin{cor} \label{c6} If $Z(\sC)$ is a field, consider the following conditions:
\begin{thlist}
\item The functor $\bar\pi$ of \eqref{eq4} is an equivalence of $\otimes$-categories.
\item  $Z(T(\sC))$ is a field. 
\item $\sC\to T(\sC)$ is faithful.
\item $\sN=0$.
\end{thlist}
Then (i) $\iff$ (ii), and (i) + (iii) $\iff$ (iv).
\end{cor}

\begin{proof} (i) $\iff$ (ii) is obvious in view of Proposition \ref{l3}.
The implication (i) + (iii) $\Rightarrow$ (iv) is also trivial. Assume (iv). Then $\sC^\natural$ is abelian and split. By Corollary \ref{r1}, the functor $\sC^\natural\to T(\sC^\natural)$ is an equivalence of $\otimes$-categories. By Theorem \ref{t1}, $T(\sC)\to T(\sC^\natural)$ is also an equivalence of $\otimes$-categories. This implies (i), and obviously (iii) as well.
\end{proof}

\begin{ex}\label{ex1} Suppose $\car K=0$. Let $\sC=\Rep_K(G)$ where $G$ is an affine $K$-group. In Example \ref{ex3}, we may take $L=K$, $\sV=\Vec_K$ and for $H$ the forgetful functor. Here $\sC$ and $\sC/\sN$ are abelian. The functor $\bar\pi$ from \eqref{eq4} and the $\otimes$-retraction $\rho_\sC$ of Corollary \ref{r1} yield an exact $\otimes$-functor
\begin{equation}\label{eq9}
T(\sC)\to S(\sC)\times \sC.
\end{equation} 
If $G$ is proreductive, then $\sN=0$, $\sC=T(\sC)=S(\sC)$ and \eqref{eq9} factors through the diagonal functor. On the other hand, if $G$ is not proreductive, then $\sN\ne 0$ and $\bar \pi$ \emph{does not} factor through the retraction $\rho_\sC$.\\ 
If $G=\G_a$ (or more generally if its prounipotent radical is $\G_a$, as in \cite[App. C]{AK2}), O'Sullivan has proven that \eqref{eq9} is an equivalence \cite{os4}. Can one compute $T(\sC)$ in more complicated cases?
\end{ex}

\section{Application to motivic conjectures}\label{s8}

\subsection{The standard conjecture D}
\begin{thm}\label{t3} Let $\sM_\rat(k)$ be the $\Q$-linear category of Chow motives over a field $k$.\\
a) Let $\pi:\sM_\rat(k)\to \sM_\num(k)$ be the canonical functor to the $\Q$-linear  abelian category of motives modulo numerical equivalence. The functor $\pi$ induces a $\otimes$-functor 
\begin{equation}\label{eq7}
\bar \pi: T(\sM_\rat(k))\to \sM_\num(k)\end{equation}
which is a Serre localisation.\\
b) Let $F:\sM_\rat(k)\to \sA$ be a $\otimes$-functor where $\sA$ is abelian and rigid.
If $\sA$ is split and $F$ is full, then $F$ factors through numerical equivalence.\\ 
c) We have: $\bar\pi$ is an equivalence $\iff$ $Z(T(\sM_\rat(k))$ is a field.
\end{thm}

\begin{proof} Note that $Z(\sM_\rat(k))=\Q$.\\ 
a)  In Example \ref{ex3}, take $\sC=\sM_\rat(k)$ and for $H$ a Weil cohomology. Here, $\sV=\Vec_L^\pm$, the abelian $\otimes$-category of finite-dimensional $\Z/2$-graded $L$-vector spaces, where $L$ is the field of coefficients of $H$; the commutativity constraint is given by the Kozsul rule.   We apply Theorem \ref{t2}. The category $S(\sC)= (\sC/\sN)^\natural$ is the category $\sM_\num(k)$ (\cite{jannsen}, which inspired Example \ref{ex3}).

b) follow from Corollary \ref{c2}.

d) follows from Corollary \ref{c6}.
\end{proof}

\begin{cor}\label{c1}
 If $Z(T(\sM_\rat(k))$ is a field, the standard conjecture D is true.\qed
\end{cor}

In Theorem \ref{t3}, we can replace rational equivalence by any adequate equivalence relation which is coarser than the homological equivalence given by a Weil cohomology $H$ as in the proof of a). For example, this homological equivalence itself yields the category of motives $\sM_H(k)$ together with the faithful functor $H:\sM_H(k)\to \Vec_L^\pm$. Applying Corollary \ref{c6}, we find:

\begin{thm}\label{t9}
The standard conjecture D holds for the Weil cohomology $H$ $\iff$ $Z(T(\sM_H(k))$ is a field.\qed
\end{thm}

In \cite{schappi}, Sch\"appi constructs a graded-Tannakian category $\M_H (k)$ along with a $\otimes$-functor $S:\sM_H(k)\to \M_H(k)$ lifting $H$, and proves that $S$ is an equivalence of categories under the standard conjecture D. From Theorem \ref{t1}, we obtain a $\otimes$-functor $$\bar S:T(\sM_H (k))\to \M_H(k)$$ extending $S$. Using $\bar S$, we get:

\begin{prop} \label{p8} The standard conjecture D holds for $H$ if and only if $\M_H(k)$ is split (\eg semi-simple) and $S$ is full.
\end{prop}

\begin{proof} \emph{If}: by \cite[Th. 3.2.1 (i)]{schappi}, the standard conjecture D implies that $S$ is an equivalence of categories. Since it also implies that $\sM_H(k)$ is abelian semi-simple, we conclude. \emph{Only if}: by Remark \ref{r2}, the hypothesis implies that $\bar S$ factors through $S(\sM_H(k))$;  this in turn implies that $H$ factors through numerical equivalence.
\end{proof}

(See \cite[Prop. 3.3.4]{schappi} for a consequence of the semi-simplicity of $\M_H(k)$, assumed alone.)

\subsection{Voevodsky's conjecture}

Let $k$ be a field. Recall that an algebraic cycle $z$ on a smooth projective variety $X$ is \emph{smash-nilpotent} if $z^{\times n}$ is rationally equivalent to $0$ on $X^n$ for $n\gg 0$. This defines an adequate equivalence relation $\tnil$.  In  \cite[Conj. 4.2]{voe}, Voevodsky conjectured that   $\sM_\tnil(k)\iso \sM_\num(k)$.

By Theorem \ref{t1}, we have $T(\sM_\rat(k))\iso T(\sM_\tnil(k))$. Let us write $T(k)$ for this rigid abelian $\otimes$-category.

\begin{prop}\label{p5} Voevodsky's conjecture is equivalent to the following two statements put together:
\begin{thlist}
\item $Z(T(k))$ is a field;
\item the functor $\sM_\tnil(k)\to T(k)$ is faithful.
\end{thlist}
\end{prop}

\begin{proof}
This follows from Corollary \ref{c6} applied to $\sC=\sM_\tnil(k)$, noting that this category is pseudo-abelian by definition.\end{proof}

\subsection{Motives over a base}

Let $S$ be a nonempty, connected, separated regular excellent scheme of finite Krull dimension. We then have the Deninger-Murre--O'Sullivan rigid $\otimes$-category of relative Chow motives over $S$ \cite{den-murre}, \cite[\S 5.1]{os3}. For coherence with the classical definition of motives, we shall restrict to the thick subcategory of O'Sullivan's category defined by motives of smooth projective $S$-schemes (O'Sullivan considers more generally smooth proper $S$-schemes): we denote this category by $\sM_\rat(S)$. We have $Z(\sM_\rat(S))=\Q$.  

Let $K$ be the function field of $S$. If $j:\Spec K\to S$ is the corresponding inclusion, we have the restriction $\otimes$-functor
\begin{equation*}
j^*:\sM_\rat(S)\to \sM_\rat(K).
\end{equation*}

We write $\sM_\rat(K,S)$ for its essential image (motives with good reduction relatively to $S$).

\begin{thm}\label{t6}  The functor $j^*$ is full and its kernel is smash-nilpotent. It induces an equivalence of categories
\[T(S)\iso T(K,S)\]
where $T(S):=T(\sM_\rat(S))$ and $T(K,S):=T(\sM_\rat(K,S))$, and  a full embedding
\[\sM_\num(S)\inj \sM_\num(K)\]
where $\sM_\num(S):=(\sM_\rat(S)/\sN)^\natural$.
\end{thm}

\begin{proof} The first assertions are \cite[Prop. 5.1.1]{os3}. The equivalence of categories then follows fromTheorem \ref{t1}, while the full embedding follows from  Lemma \ref{l8}.
\end{proof}

Let now $i:Z\inj S$ be a closed subscheme of $S$, also connected and regular, with function field $L$; we have a pull-back functor $i^*:\sM_\rat(S)\to \sM_\rat(Z)$. Theorem \ref{t6} then yields a ``specialisation'' functor
\[i^!:T(K,S)\to T(L,Z).\]

Since $i^*$ is not full, it does not \emph{a priori} induce a functor $\sM_\num(S)\to \sM_\num(Z)$. 

Using a Weil cohomology $H$ verifying the smooth and proper base change theorem (e.g. $l$-adic cohomology for a prime $l$ invertible on $S$) and using the monoidal section theorem, we can construct as in \cite[Th. 11]{AK} a ``specialisation'' $\otimes$-functor $\sM_\num(S)\to \sM_\num(Z)$, depending \emph{a priori} on $H$; we leave details to the interested reader.

\subsection*{Acknowledgements}
We would like to thank Y. Andr\'e for his kind interest in this work and for pointing out Example \ref{ex1} at an early stage, and  J. Wildeshaus for the reference to Prop. 5.1.1 in \cite{os3}. The first author acknowledges K. Coulembier for explaining his works \cite{coul1} and \cite{coul2}. The second author thanks P. O'Sullivan for several exchanges on his work \cite{os} and for communicating \cite{os4}. Finally, the first author thanks Institut de Math\'ematiques de Jussieu-Paris Rive Gauche for hospitality and both authors thank CNRS for support.

\end{document}